\documentclass[12pt]{amsart}
\usepackage{amsmath, amsthm, amscd, amsfonts, amssymb, graphicx, color}
\usepackage{hyperref}
\usepackage{float}
\usepackage[capitalize,nameinlink]{cleveref}
\newtheorem{thm}{Theorem}
\newtheorem{rmk}{Remark}
\newtheorem{theorem}{Theorem}[section]
\newtheorem{lemma}[theorem]{Lemma}

\newtheorem{corollary}[theorem]{Corollary}
\newtheorem{definition}[theorem]{Definition}
\theoremstyle{definition}

\newtheorem{remark}[theorem]{Remark}
\numberwithin{equation}{section}
\showoutput
\begin{document}
\setcounter{page}{1}

\title[Quasi-Einstein metrics and a curvature identity associated with the Ricci flow]{Quasi-Einstein metrics and a curvature identity associated with the Ricci flow}

\author[Atreyee Bhattacharya and Sayoojya Prakash]{Atreyee Bhattacharya and Sayoojya Prakash}
\begin{abstract}
Quasi-Einstein manifolds are well-studied generalizations of Einstein manifolds. This includes gradient Ricci solitons and has a natural correspondence with the warped product Einstein manifolds. A quasi-Einstein metric is said to be rigid when it reduces to an Einstein metric. On a different note, Einstein metrics can be viewed as fixed points of the Ricci flow up to homothety. While gradient Ricci solitons are generalized fixed points of the Ricci flow, not much is known, in general, about the evolution of quasi-Einstein metrics under the Ricci flow. In this paper, we employ an identity associated to the evolution of curvature along the Ricci flow, to conclude the rigidity of certain closed quasi-Einstein manifolds.   
\end{abstract}
\maketitle
\section{Introduction} 

The Ricci flow, a geometric evolution equation of Riemannian manifolds, analogous to the heat equation, was introduced by R. Hamilton (\cite{TRH}). Given a Riemannian manifold $(M,g_{0})$, the Ricci flow (see \cite{BCI} for details) starting at $g_{0}$ is the smooth one-parameter family of metrics $\{g(t)\}_{t\in[0,T]}$ on $M$, satisfying the following partial differential equation:
\begin{equation}\label{RFE}
\frac{\partial g(t)}{\partial t} = -2Ric_{g(t)}, \ \ g(0)=g_0
\end{equation}
where for each $t\in[0,T]$, $Ric_{g(t)}$ denotes the Ricci tensor of the metric $g(t)$. A Riemannian metric is said to be \textit{Einstein} if its Ricci tensor is a constant multiple of the metric. It follows that under the Ricci flow, Einstein metrics remain unaltered up to homothety: In fact, if $(M,g)$ is an Einstein manifold with $Ric_g =\lambda g$, the solution to \cref{RFE} with initial condition $g(0) = g$ is given by $g(t) = (1-2\lambda t)g$. It turns out that all fixed points of the Ricci flow up to homothety, are Einstein. More generally, Ricci solitons, popularly known as the \textit{self-similar solutions} or \textit{generalized fixed points} of the Ricci flow that generalize Einstein manifolds, have been extensively studied both in theoretical physics (see \cite{RSTP}) and differential geometry and remain unaltered by the Ricci flow on the moduli space of Riemannian metrics (see \cite{BCRF}, for details). G. Perelman  (\cite{PRF}) showed that all closed Ricci solitons belong to a special class of Riemannian manifolds, called \textit{quasi-Einstein}, another well-known generalization of Einstein manifolds. More precisely, a Riemannian manifold $(M,g)$ is called quasi-Einstein, if there exists a function $f \in C^{\infty}(M)$ such that 
$$Ric_{g} + Hess_{g}(f) - \frac{1}{m}df\otimes df = \lambda g,$$ for some  $\lambda \in \mathbb{R}$ and $0 < m \leq \infty$. It can be checked that a quasi-Einstein manifold $(M,g,f,\lambda)$ coincides with a Ricci soliton when $m = \infty$ and it corresponds to a warped product Einstein metric when $m \in \mathbb{N}$ (see \cite{DYKWP}). Unlike the evolution of Ricci solitons under the Ricci flow, the evolution of arbitrary quasi-Einstein manifolds under the Ricci flow, is not so well understood. In this paper, we study a curvature identity derived from the Ricci flow and prove that a closed quasi-Einstein manifold satisfying the identity does not get perturbed by a Ricci flow and hence, reduces to an Einstein manifold. 

A quasi-Einstein metric (or a Ricci soliton) is called \textit{rigid} if it reduces to an Einstein metric. Note that in the non-compact set-up, Ricci solitons are not necessarily quasi-Einstein (\cite{NCRS}). Rigidity of quasi-Einstein manifolds and Ricci solitons (not necessarily closed) have been important topics in Riemannian geometry for the past three decades. R.~Hamilton (\cite{RHS}) established that two-dimensional closed Ricci solitons are rigid. T.~Ivey proved a similar rigidity result for all closed Ricci solitons in three dimension.  More generally, he proved the rigidity of closed Ricci solitons with $\lambda \leq 0$ in all dimensions (\cite{RSTV}). Later, J.~Case, Y.~Shu, and G.~Wei proved that two-dimensional closed quasi-Einstein manifolds with finite $m$ are rigid, thereby settling the rigidity problem for all closed quasi-Einstein manifolds in dimension two. They further showed that closed K\"ahler quasi-Einstein manifolds with finite $m$ are rigid in all dimensions (\cite{RQE}). Moreover, D.S.~Kim and Y.H.~Kim (\cite{KKWP}) proved the rigidity of arbitrary closed quasi-Einstein manifolds with finite $m$ and $\lambda \leq 0$. P.~ Petersen and W.~Wylie established various rigidity results for Ricci solitons in a series of papers. In particular, they showed that a closed \textit{shrinking} gradient Ricci soliton is rigid if $\int_{M}Ric(\nabla f, \nabla f) dV_{g} \leq 0$(\cite{PWRGRS}). They also proved the rigidity of homogeneous gradient Ricci solitons and ruled out the existence of non-compact cohomogeneity one shrinking gradient Ricci solitons with non-negative curvature (\cite{PWGRSS}). Furthermore, they classified all simply connected shrinking Ricci solitons with vanishing Weyl tensor and Ricci tensor satisfying a weak integral condition up to isometry, as one of $\mathbb{S}^{n}$, $\mathbb{S}^{n-1} \times \mathbb{R}$, and $\mathbb{R}^{n}$ (\cite{CGSPW}). In dimension $4$, A.~Naber classified all non-compact shrinking solitons with bounded nonnegative curvature operator to be isometric to one of $\mathbb{R}^{4}$, finite quotients of $\mathbb{S}^{2} \times \mathbb{R}^{2}$, and $\mathbb{S}^{3} \times \mathbb{R}$(\cite{RSAN}). G.~Catino, C.~Mantegazza, L.~Mazzieri, and M.~Rimoldi (\cite{LCFQEM}) proved that complete locally conformally flat quasi-Einstein manifolds of dimension $n \geq 3$ are locally a warped product of $(n-1)$- dimensional fibres of constant curvature. However, the rigidity of a closed quasi-Einstein metric with $\lambda>0$ is yet to be fully resolved. In \cite{ABSP}, the authors obtained rigidity criteria for certain closed quasi-Einstein manifolds with $\lambda >0$ using techniques of conformal submersion. While several techniques of geometric analysis have been employed to examine the rigidity of both quasi-Einstein metrics and Ricci solitons so far, apart from \cite{RHS} and \cite{RSTV}, the authors are unaware of any other work involving the Ricci flow methods to investigate the same. In this paper, we use a curvature identity associated to the Ricci flow to explore this rigidity further.

Recall that under the Ricci flow \ref{RFE}, the Riemann curvature operator $\mathcal{R}=$ $\mathcal{R}(t): \wedge^2 T M \rightarrow \wedge^2 T M$ of $g(t)$ evolves by
\begin{equation} \label{CE}
\frac{\partial \mathcal{R}}{\partial t}=\Delta \mathcal{R} + Q(\mathcal{R})    
\end{equation}
where $Q(\mathcal{R})=\mathcal{R}^2+\mathcal{R}^{\#}$ is a homogeneous quadratic polynomial in the components of $\mathcal{R}$. As an immediate consequence of the curvature evolution \cref{CE}, it follows that the curvature operator of an Einstein manifold $(M,g,\lambda)$ satisfies the following \textit{curvature identity}:
\begin{equation}\label{Eq11}
\Delta \mathcal{R} + Q(\mathcal{R}) = 2\lambda \mathcal{R},
\end{equation} 
(see \cite{BCI}, \cite{PTLRF} for more details about \cref{Eq11}).
It is clear that if $g$ is a fixed point of \cref{RFE}, then the corresponding $\mathcal{R}$ is a fixed point of \cref{CE}, both up to homothety. This motivates us to ask if and when the converse is true. \textit{More generally, if the Riemann curvature operator $\mathcal{R}$ of a manifold $(M,g)$, is a fixed point of \cref{CE} up to a conformal factor, and in particular it satisfies the following equation at each point,
\begin{equation}\label{CI1}
\Delta \mathcal{R} + Q(\mathcal{R}) = h \mathcal{R}
\end{equation} 
for some smooth function $h \in C^{\infty}(M)$, then one would like to know when we can conclude $g$ to be a fixed point of \cref{RFE} upto homothety or equivalently, $g$ is Einstein.} We refer to \cref{CI1} as the \textit{Curvature identity}.\\ It is straightforward to see that a complete Riemannian manifold whose universal cover is a Riemannian product of non-flat Einstein manifolds satisfying the Curvature identity itself reduces to an Einstein manifold. However, a Riemannian product of Einstein manifolds where one of them is flat also satisfies the Curvature identity. It is natural to seek more Riemannian manifolds that satisfy the Curvature identity. We first observe the following.
\begin{rmk}\label{RS}
It is easy to check that a closed manifold $(M^{n},g)$ satisfying the Curvature identity for $h = \frac{2s}{n}$, where $s$ is the scalar curvature (not necessarily a constant) of $(M^{n},g)$, reduces to an Einstein manifold. In fact, taking the trace of the Curvature identity yields,
\begin{equation}\label{TCI}
\Delta s + 2||Ric_{g}||^{2} = \frac{2s^{2}}{n}
\end{equation}
Using the inequality $||Ric_{g}||^{2} \geq \frac{s^{2}}{n}$ and applying the maximal principle in \cref{TCI}, it follows that the equality holds. Consequently,  $Ric_{g} = \frac{s}{n}g$, i.e., $(M,g)$ is an Einstein manifold. The same conclusion holds for any complete Riemannian manifold for which the scalar curvature attains a global minimum.  
\end{rmk}
It follows from the above discussion that in order to construct non-trivial Riemannian manifolds satisfying the Curvature identity, one needs to choose both the manifold and the conformal factor $h$ carefully. Since Ricci solitons are generalized fixed points of the Ricci flow, one would like to know in particular if there are Ricci solitons satisfying the Curvature identity.
More generally, one would also like to pose this question for all quasi-Einstein manifolds. It turns out that in a non-compact setup, one can construct non-rigid quasi-Einstein examples satisfying the Curvature identity.
\begin{rmk}
     We illustrate the previous statement with the following examples, both constructed on the product manifold $M = \mathbb{R} \times N$, equipped with the product metric $g$ where $(N,g_{N})$ is an Einstein manifold with Einstein constant $\lambda$.
    \begin{enumerate}
        \item If  $\lambda < 0$, then $(M,g,f,m,\lambda)$ is a quasi-Einstein manifold where $f : M \rightarrow \mathbb{R}$ is defined as $f(t,x) = - m\ln({\cosh({\sqrt{\frac{-\lambda}{m}}t}}))$, where $t \in \mathbb{R}$ and $x \in N$. 
        \item If $\lambda \neq 0$, then $(M,g,f,\lambda)$ is a Ricci soliton where 
        $f : M \rightarrow \mathbb{R}$ is defined as $f(t,x) = \frac{\lambda t^{2}}{2}$, where $t \in \mathbb{R}$ and $x \in N$.
    \end{enumerate}
    In both cases, it is straightforward to verify that $(M,g)$ satisfies the Curvature identity for $h = 2\lambda$.
\end{rmk}
In light of this, we will restrict our discussion to closed quasi-Einstein manifolds. To this end, we prove the following result.
\begin{thm}\label{MF}
Let $(M,g,f,m,\lambda)$ be a closed quasi-Einstein manifold satisfying the Curvature identity. Then $\int ((h-2\lambda)s)dV_{g} \geq 0$ and $(M,g,f)$ is rigid if and only if $\int ((h-2\lambda)s)dV_{g} = 0$, where $s$ is the scalar curvature of $(M,g)$.
\end{thm}
\noindent As an immediate consequence of \cref{MF}, we have the following result.
\begin{thm}\label{CF}
   Let $(M,g,f,m,\lambda)$ be a closed quasi-Einstein manifold satisfying the Curvature identity. If a global maximum of $h$ coincides with either a local maximum or a local minimum point of $f$, then $(M,g,f)$ is rigid. In particular, if $h$ is a constant function, then $g$ is rigid.
\end{thm}
\begin{rmk}
It is worth noting that there exist quasi-Einstein manifolds for which $\int((h -2\lambda)s)dV_{g} > 0$ (see \cref{R34} for details).
\end{rmk}
\section{Preliminaries}\label{S2}
   In this paper, we will consider closed, connected, oriented Riemannian manifolds. 
\subsection{Riemannian geometry}
\noindent In this section, we recall some fundamental definitions in Riemannian geometry and establish the notation to be followed (see \cite{RGDC}, \cite{RQE} for more details).
\begin{definition}
    Given a Riemannian manifold $(M,g)$, its \textbf{Riemann curvature tensor $R$} is a $(0,4)$ tensor defined by,
    $$R(X,Y,Z,W) = g(\nabla_{Y}\nabla_{X}Z - \nabla_{X}\nabla_{Y}Z + \nabla_{[X,Y]}Z,W)$$
     where $X$, $Y$, $Z$ and $W$ are arbitrary vector fields defined on an open subset of $M$ and $\nabla$ is the Levi-Civita connection of $(M,g)$.
     \end{definition}

\begin{definition}
A Riemannian manifold $(M,g)$ is said to be
\begin{enumerate} 
\item \textbf{Einstein} if its Ricci curvature is a scalar multiple of the Riemannian metric, i.e., $Ric_{g} = \lambda g \ \ \text{ for some } \ \lambda \in \mathbb{R}.$
\item a \textbf{Ricci soliton} if there exists a vector field $X \in \chi(M)$ such that $$Ric_{g} + \frac{1}{2}L_{X}g = \lambda g,$$  where $\lambda \in \mathbb{R}$  and $L_{X}g$ is the Lie-derivative of $g$ with respect to $X$. In particular, if $X = \nabla f$ for some $f \in C^{\infty}(M)$, then $(M,g)$ is called a \textbf{gradient Ricci soliton}.
\item a \textbf{quasi-Einstein manifold} if there exists a function $f \in C^{\infty}(M)$ such that
$$Ric_{g} + Hess_{g}(f) - \frac{1}{m}df\otimes df = \lambda g$$
where  $\lambda \in \mathbb{R}$, $0 < m \leq \infty$ and $Hess_{g}(f)(X,Y) = g(\nabla_{X}\nabla f,Y)$.
\end{enumerate}
\end{definition}
\begin{remark}
\begin{enumerate}
    \item Note that a closed Ricci soliton is a gradient Ricci soliton (\cite{PRF}).  When $m = \infty$, then a quasi-Einstein manifold $(M,g,f)$ becomes a gradient Ricci soliton.
    \item For $m \in \mathbb{N}$, D.S.~Kim and Y.H.~Kim showed that $(M,g,f)$ corresponds to a warped product Einstein metric (see \cite{KKWP}).
\end{enumerate}
\end{remark}

\begin{definition}
A quasi-Einstein manifold $(M,g,f,m, \lambda)$ is said to be rigid if it reduces to an Einstein manifold. 
\end{definition}
\begin{remark}
It is easy to see that a closed, oriented, connected quasi-Einstein manifold $(M,g,f,m,\lambda)$ is rigid if and only if $f$ is constant. This is not true in general for non-compact manifolds. A complete finite $m$ quasi-Einstein manifold $(M,g,f,m,\lambda)$ is rigid if and only if $f$ is constant or $M$ is diffeomorphic to $\mathbb{R}^{n}$ with a warped product structure (see \cite{RQE}).
\end{remark}
\begin{remark}
    In this paper, the following conventions will be followed.
    \begin{enumerate}
        \item Given a vector field $X$ in $(M,g)$, the divergence of $X$  is defined pointwise as $div(X) = \sum_{i} g(\nabla _{e_{i}}X,e_{i})$ where $\{e_{i}\}$ is an orthonormal basis of the tangent space at that point.
        \item Given $f \in C^{\infty}(M)$, $\Delta f = div(\nabla f)$.
        \end{enumerate}
\end{remark}
\subsection{Algebraic curvature operators}
\noindent We now recall some relevant concepts of multilinear algebra (see \cite{BCRF} for more details).\\
\noindent Let $(\mathbb{R}^{n},\langle , \rangle)$ be the Euclidean inner product space and $\Lambda ^{2}\mathbb{R}^{n}$ denote the exterior $2$-product of  $\mathbb{R}^n$ equipped with the canonical inner product induced by the Euclidean inner product, defined by
\[
\langle u \wedge v,\; w \wedge z \rangle := \det \begin{pmatrix}
\langle u, w \rangle & \langle u, z \rangle \\
\langle v, w \rangle & \langle v, z \rangle
\end{pmatrix}, \& u,v,w,z \in \mathbb{R}^{n}.
\]
It is easy to see that with this inner product, $\{e_{i} \wedge e_{j}\}_{i<j}$ becomes an orthonormal basis of $\Lambda^2 \mathbb{R}^n$ whenever $\left\{e_i\right\}_{i=1}^n$ is an orthonormal basis of $\mathbb{R}^n$. $\Lambda^2 \mathbb{R}^n$ is identified with the Lie algebra $\mathfrak{so}(n)$ via the linear map $x \wedge y: \mathbb{R}^n \rightarrow \mathbb{R}^n$ defined by $(x \wedge y): z \mapsto\langle y, z\rangle x-\langle x, z\rangle y$, for $x, y, z \in \mathbb{R}^n$. With this identification $\mathfrak{so}(n)$ has the inner product $$\langle A, B\rangle=-\frac{1}{2} \operatorname{tr}(A B), \& A,B \in \mathfrak{so}(n).$$
Let $S^{2}(\mathfrak{so}(n))$ denote the space of all self-adjoint endomorphisms of $\Lambda^{2}\mathbb{R}^{n}$.\\
\noindent Given a Riemannian manifold $(M,g)$ and a point $p \in M$, the tangent space $T_{p}M$ can be identified with $\mathbb{R}^{n}$. With this identification, all the definitions mentioned above can be extended point-wise to the Riemannian manifold.
\begin{definition}
Given a Riemannian manifold $(M,g)$ and $p \in M$, its \textbf{Riemann curvature operator $\mathcal{R}$} is the symmetric bilinear form on $\Lambda ^{2} T_{p}M$ or self-adjoint endomorphism of $\Lambda ^{2} T_{p}M$ defined by,
$$\mathcal{R}(x \wedge y, z \wedge w) = \langle \mathcal{R}(x \wedge y), z \wedge w> = 2R(x,y,z,w), \  x,y,z,w \in T_pM.$$
\end{definition}

   \noindent Modelled on the properties of the Riemann curvature operator, we define the algebraic curvature operator.
\begin{definition}
    The space of \textbf{algebraic curvature operators} (in dimension $n$), denoted by $S_{B}^{2}(\mathfrak{so}(n))$, defined to be the subspace of $S^{2}(\mathfrak{so}(n))$ satisfying the following algebraic Bianchi identity.
    $$\langle R(x \wedge y), w \wedge z\rangle  + \langle R(y\wedge w),x \wedge z\rangle  + \langle R(w \wedge x), y \wedge z\rangle  = 0$$
    where $R \in S^{2}(\mathfrak{so}(n))$ and $x,y,w,z \in \mathbb{R}^{n}$.
    \end{definition}
    Since algebraic curvature operators are modelled on the Riemann curvature operator, we can similarly define the corresponding Ricci tensor and scalar curvature (refer to \cite{BCRF} for details).
    \begin{definition}
        \begin{enumerate}
            \item  For $R \in S_{B}^{2}(\mathfrak{so}(n))$, the \textbf{Ricci tensor} $Ric(R)$ is defined by
    $$\langle Ric(R)(x),y\rangle  = \sum_{k = 1}^{n} \langle R(x \wedge e_{k}), y \wedge e_{k} \rangle $$
    for all $x,y \in \mathbb{R}^{n}$ and $\{e_{k}\}$ is an orthonormal basis of $\mathbb{R}^{n}$.
    \item The \textbf{scalar curvature} $Scal(R)$ is the trace of $Ric(R)$
        \end{enumerate}
    \end{definition}
    
     Now, we will introduce an operator that is useful in defining the evolution equation of the Riemann curvature operator.
     \begin{definition}
           The \textbf{sharp operator} $\# : S^{2}(\mathfrak{so}(n)) \times S^{2}(\mathfrak{so}(n)) \rightarrow S^{2}(\mathfrak{so}(n))$ is a bilinear map defined by,
    $$\langle (A \# B)(\phi), \psi\rangle  = \sum_{\alpha, \beta} \langle [A(\omega _{\alpha}),B(\omega_{\beta})], \phi\rangle \langle [\omega_{\alpha}, \omega_{\beta}], \psi\rangle $$
    where  $\{\omega_{\alpha}\}$ is an orthonormal basis of $S^{2}(\mathfrak{so}(n))$.
     \end{definition}
   In particular, if $A = B$, we define the following operator.
   \begin{definition}
         The operator $Q : S^{2}(\mathfrak{so}(n))  \rightarrow S^{2}(\mathfrak{so}(n))$ is defined as,
    $$Q(A) = A \circ A + A \# A$$
   \end{definition}
    \begin{remark}
        Note that for an algebraic curvature operator $R$,
         $$\langle Ric(Q(R))(e_{i}),e_{j}\rangle  = \sum_{k,l}\langle  Ric(R)(e_{k}),e_{l} \rangle  \langle R(e_{i} \wedge e_{k}),e_{j} \wedge e_{l}\rangle  $$
    $$Scal(Q(R)) = \sum_{k,l} \langle Ric(e_{k}),e_{l}\rangle ^{2}$$
    \end{remark}
\section{Main Results}
\noindent 
In this section, we prove Theorems \ref{MF} and \ref{CF}. We begin by establishing the following result as a direct consequence of Theorem 1.2 of \cite{CPW} and Corollary 1.7 of \cite{CQ}.
\begin{lemma}
If $(M,g,f,m,\lambda)$ is a closed quasi-Einstein manifold with $m > 1$, harmonic Weyl tensor, and satisfying $W(\nabla f, -,- ,\nabla f) = 0$ and the Curvature identity, then $(M,g,f)$ is rigid. Moreover, a closed Bach-flat quasi-Einstein manifold $(M^{n},g,f,m,\lambda)$ with $m \neq 1$, $n \geq 4$ satisfying the Curvature identity is also rigid.
\end{lemma}
Next, we would like to prove \cref{MF}.
\begin{proof}[Proof of \cref{MF}]
    We have,
   $$\Delta s + 2||Ric||^{2} = hs$$
    \begin{eqnarray}\label{Eq312}
        (h -2\lambda)s &=&  \Delta s + 2||Ric||^{2} - 2\lambda s
    \end{eqnarray}
    For a quasi-Einstein manifold, we have
    \begin{equation}\label{Eq313}
         s + \Delta f - \frac{1}{m}||\nabla f||^{2} = \lambda n
    \end{equation}
    By substituting \cref{Eq313} in \cref{Eq312}, we get,
    \begin{eqnarray}\label{Eq314}
        (h - 2\lambda)s &=& \Delta s + 2||Ric - \frac{s}{n}g||^{2} + \frac{2}{n}(\Delta f)^{2} + \frac{2}{m^{2}n}||\nabla f||^{4} - 2\lambda \Delta f + \frac{2\lambda}{m}||\nabla f||^{2}\nonumber\\&& - \frac{4}{mn}\Delta f ||\nabla f||^{2}
    \end{eqnarray}
    By Lemma 3.2 of \cite{RQE},
    \begin{equation}\label{Eq315}
        \frac{1}{2}\Delta ||\nabla f||^{2} = ||Hess(f)||^{2} - Ric(\nabla f, \nabla f) + \frac{2}{m}||\nabla f||^{2}\Delta f
    \end{equation}
     \begin{eqnarray}\label{Eq316}
        (h - 2\lambda)s &=& \Delta s + 2||Ric - \frac{s}{n}g||^{2} + \frac{2}{n}(\Delta f)^{2} + \frac{2}{m^{2}n}||\nabla f||^{4} - 2\lambda \Delta f + \frac{2\lambda}{m}||\nabla f||^{2}\nonumber\\&& + \frac{2}{n}||Hess(f)||^{2} - \frac{2}{n}Ric(\nabla f, \nabla f) - \frac{1}{n}\Delta ||\nabla f||^{2}
    \end{eqnarray}
     By Bochner's formula (\cite{HRFBC}),
    \begin{equation}\label{Eq317}
        \frac{1}{2}\Delta ||\nabla f||^{2} = ||Hess(f)||^{2} + Ric(\nabla f, \nabla f) + g(\nabla f, \nabla \Delta f)
    \end{equation}
    Using \cref{Eq317} in \cref{Eq316},
     \begin{eqnarray}\label{Eq318}
        (h - 2\lambda)s &=& \Delta s + 2||Ric - \frac{s}{n}g||^{2} + \frac{2}{n}(\Delta f)^{2} + \frac{2}{m^{2}n}||\nabla f||^{4} - 2\lambda \Delta f + \frac{2\lambda}{m}||\nabla f||^{2}\nonumber\\&& + \frac{4}{n}||Hess(f)||^{2} + \frac{2}{n}g(\nabla f, \nabla \Delta f) - \frac{2}{n}\Delta ||\nabla f||^{2}
    \end{eqnarray}
    Using the divergence theorem in \cref{Eq318},
    \begin{eqnarray}\label{Eq319}
        \int  ((h - 2\lambda)s) dV_{g} &=&  2\int ||Ric - \frac{s}{n}g||^{2} dV_{g}  + \frac{2}{m^{2}n}\int ||\nabla f||^{4} dV_{g} + \frac{2\lambda}{m}\int
        ||\nabla f||^{2} dV_{g}\nonumber\\&& + \frac{4}{n}\int||Hess(f)||^{2}dV_{g} 
    \end{eqnarray}
    Using \cref{Eq319}, $ \int  ((h - 2\lambda)s) dV_{g} \geq 0$ and  and $(M,g,f)$ is rigid if and only if $\int ((h-2\lambda)s)dV_{g} = 0$.
\end{proof}
\begin{remark}\label[remark]{R34}
  It is important to note that there are quasi-Einstein manifolds that fail to satisfy the above theorem. Wang and Zang proved that for a compact Riemann surface $(M,g)$ with positive first Chern class and a non-vanishing Futaki invariant, $\int ((h -2\lambda)s)dV_{g} > 0$ and there exists a $X \in \chi(M)$ such that $(M,g,X)$ is a non-rigid Ricci soliton(see \cite{EXK} for details).
\end{remark}
Analogous to Lemma 2.1 of \cite{CGSPW}, we obtained the following result for a quasi-Einstein manifold which will be used in \cref{CF}.
\begin{lemma}\label[lemma]{LRL}
    If $(M,g,f,m,\lambda)$ is a quasi-Einstein manifold, then
    {\small
   \begin{align}
       (\Delta R)(X,Y,Z,W) + \frac{1}{2}Q(\mathcal{R})(X,Y,Z,W) =& 2\lambda R(X,Y,Z,W) + (\nabla_{\nabla f}R)(X,Y,Z,W)\nonumber\\&+ \frac{1}{m}(R(\nabla f,Y,Z,W)X(f) + R(X,\nabla f,Z,W)Y(f)\nonumber\\&+R(X,Y,\nabla f,W)Z(f) + R(X,Y,Z,\nabla f)W(f))\nonumber\\&+\frac{2}{m}(Hessf(Y,Z)Hessf(X,W)\nonumber\\&- Hessf(X,Z)Hessf(Y,W))
   \end{align}} 
   where $R$ and $\mathcal{R}$ denote the Riemann curvature tensor and Riemann curvature operator of $(M,g)$ respectively and $X,Y,Z,W \in \chi(M)$.
\end{lemma}
\begin{proof}
    Let $p \in M$ and $X,Y,Z,W$ be vector fields on $M$ such that $\nabla X|_{p} = \nabla Y|_{p} = \nabla Z|_{p} = \nabla W|_{p} = 0$. Let $\{E_{i}\}$ be a normal basis at $p$.
    {\small
    \begin{align}\label{LRE}
        (\Delta R)(X,Y,Z,W)=&\sum_{i}(\nabla_{E_{i}}\nabla_{E_{i}}R)(X,Y,Z,W)\nonumber\\=&\sum_{i}(\nabla_{E_{i}}\nabla_{X}R)(E_{i},Y,Z,W) - \sum_{i}(\nabla_{E_{i}}\nabla_{Y}R)(E_{i},X,Z,W)\nonumber\\=&\sum_{i}(R_{X,E_{i}}R)(E_{i},Y,Z,W) - \sum_{i}(R_{Y,E_{i}}R)(E_{i},X,Z,W)\nonumber\\&+ \sum_{i}(\nabla_{X}\nabla_{E_{i}}R)(E_{i},Y,Z,W) - \sum_{i}(\nabla_{Y}\nabla_{E_{i}}R)(E_{i},X,Z,W)
    \end{align}}
    It is easy to see that,
{\small
    \begin{align}\label{CCE}
    \sum_{i}(R_{X,E_{i}}R)(E_{i},Y,Z,W) - \sum_{i}(R_{Y,E_{i}}R)(E_{i},X,Z,W)=& R(Ric(X),Y,Z,W) + R(X,Ric(Y),Z,W)\nonumber\\
    &+ 2(B(X,Y,W,Z) + B(X,W,Y,Z) \nonumber\\& - B(X,Y,Z,W) - B(X,Z,Y,W))\nonumber\\=& R(Ric(X),Y,Z,W) + R(X,Ric(Y),Z,W)\nonumber\\&-\frac{1}{2}Q(\mathcal{R})(X,Y,Z,W)
    \end{align}}
    where $B(X,Y,W,Z) = \langle R(X,-,Y,-),R(W,-Z,-)\rangle$.
    {\small
    \begin{align}\label{GRX}
        \sum_{i}(\nabla_{X}\nabla_{E_{i}}R)(E_{i},Y,Z,W) =&\sum_{i}(\nabla_{X}\nabla_{E_{i}}R)(Z,W,E_{i},Y)\nonumber\\=&-\sum_{i}(\nabla_{X}\nabla_{Z}R)(W,E_{i},E_{i},Y)-\sum_{i}(\nabla_{X}\nabla_{W}R)(E_{i},Z,E_{i},Y)\nonumber\\=&(\nabla_{X}\nabla_{Z}Ric)(Y,W) - (\nabla_{X}\nabla_{W}Ric)(Y,Z)\nonumber\\=& X((\nabla_{Z}Ric)(Y,W)- (\nabla_{W}Ric)(Y,Z))
    \end{align}}
Using $Ric + Hessf - \frac{1}{m}df \otimes df = \lambda g$ in \cref{GRX},
{\small
\begin{align}\label{GRXSE}
    \sum_{i}(\nabla_{X}\nabla_{E_{i}}R)(E_{i},Y,Z,W) =&X((\nabla_{W}Hessf)(Y,Z)-(\nabla_{Z}Hessf)(Y,W))\nonumber\\& + \frac{1}{m}X((\nabla_{Z}df \otimes df)(Y,W) - (\nabla_{W}df \otimes df)(Y,Z))\nonumber\\=&(\nabla_{X}R)(\nabla f,Y,Z,W) + R(\nabla_{X}\nabla f,Y,Z,W)\nonumber\\& + \frac{1}{m}X((\nabla_{Z}df \otimes df)(Y,W) - (\nabla_{W}df \otimes df)(Y,Z))\nonumber\\=&(\nabla_{X}R)(\nabla f,Y,Z,W) + \lambda R(X,Y,Z,W)\nonumber\\&-R(Ric(X),Y,Z,W) + \frac{1}{m}X(f)R(\nabla f,Y,Z,W)\nonumber\\& + \frac{1}{m}X((\nabla_{Z}df \otimes df)(Y,W) - (\nabla_{W}df \otimes df)(Y,Z))\nonumber\\=&(\nabla_{X}R)(\nabla f,Y,Z,W) + \lambda R(X,Y,Z,W)\nonumber\\&-R(Ric(X),Y,Z,W) + \frac{1}{m}X(f)R(\nabla f,Y,Z,W)\nonumber\\&+ \frac{1}{m}((\nabla_{X}Hessf)(Y,Z)W(f) - (\nabla_{X}Hessf)(Y,W)Z(f))\nonumber\\&+\frac{1}{m}(Hessf(Y,Z)Hessf(X,W) - Hessf(Y,W)Hessf(X,Z))
\end{align}}
Similarly, we obtain
{\small
\begin{align}\label{GRYE}
    \sum_{i}(\nabla_{Y}\nabla_{E_{i}}R)(E_{i},X,Z,W)=&(\nabla_{Y}R)(\nabla f,X,Z,W) - \lambda R(X,Y,Z,W)\nonumber\\&-R(Ric(Y),X,Z,W) + \frac{1}{m}Y(f)R(\nabla f,X,Z,W)\nonumber\\&+ \frac{1}{m}((\nabla_{Y}Hessf)(X,Z)W(f) - (\nabla_{Y}Hessf)(X,W)Z(f))\nonumber\\&+\frac{1}{m}(Hessf(X,Z)Hessf(Y,W) - Hessf(X,W)Hessf(Y,Z))
\end{align}}
Using \cref{LRE}, \cref{CCE}, \cref{GRXSE} and \cref{GRYE}, we obtain the result.
\end{proof}
As an immediate consequence, we obtain the following result.
\begin{lemma}\label[lemma]{LF}
    If $(M,g,f,m,\lambda)$ is a closed quasi-Einstein manifold  that satisfies the Curvature identity, then
    \begin{equation}
        \begin{aligned}
            (h - 2\lambda)s = & \frac{4}{m}Ric(\nabla f , \nabla f) + \frac{2}{m} ( ||Hessf||^{2} - (\Delta f)^{2}) + g(\nabla f, \nabla s)\nonumber
        \end{aligned}
    \end{equation}
    where $s$ is the scalar curvature of $(M,g)$ respectively.
\end{lemma}
\begin{proof}
        Using \cref{LRL} and the Curvature identity of $(M,g)$,
        {\small
        \begin{align}\label{Eq34}
            (h-2\lambda)R(X,Y,Z,W) = &\frac{X(f)}{m}R(\nabla f,Y,Z,W) + \frac{Y(f)}{m}R(X,\nabla f,Z,W)\nonumber\\ +& \frac{Z(f)}{m}R(X,Y,\nabla f,W)  + \frac{W(f)}{m}R(X,Y,Z,\nabla f)\\  +& \frac{2}{m}(Hessf(X,W)Hessf(Y,Z)-Hessf(X,Z)Hessf(Y,W))\nonumber \\ +&  (\nabla_{\nabla f}R)(X,Y,Z,W)\nonumber     
            \end{align}}
    By taking the trace of \cref{Eq34} in the second and fourth entries at $p \in M$,
    \allowdisplaybreaks
    {
    {\small
        \begin{align}\label{Eq35}
            (h-2\lambda)r(X,Z) = & \frac{X(f)}{m}r(\nabla f,Z) + \frac{2}{m}R(X,\nabla f,Z,\nabla f) + \frac{Z(f)}{m}r(\nabla f,X)\nonumber\\& + \frac{2}{m}(\sum_{i}Hessf(X,X_{i})Hessf(Z,X_{i}) - Hessf(X,Z)\Delta f)\\& + \sum_{i}(\nabla_{\nabla f}R)(X,X_{i},Z,X_{i})\nonumber
       \end{align}}}
       where $\{X_{i}\}$ is a normal basis at $p$.\\
    We obtain the result by taking the trace of \cref{Eq35}.
\end{proof}
\noindent Using \cref{LF}, we are now ready to prove \cref{CF}.
  \begin{proof}[Proof of \cref{CF}]
 Using \cref{LF}, we can conclude that $h_{max} \leq 2\lambda$ if $f$ attains maxima or minima at a point of global maxima of $h$. Now, the result follows from \cref{MF}.
       \end{proof}
As immediate consequences, we obtain the following results.
\begin{corollary}
 Products of quasi-Einstein manifolds satisfying the conditions of \cref{CF} either reduce to an Einstein manifold or to a Riemannian product of a flat manifold and an Einstein manifold.   
\end{corollary}
\begin{corollary}
    If $(M,g,f,m,\lambda)$ is a quasi-Einstein manifold, then $\Delta \mathcal{R} + Q(\mathcal{R}) = h\mathcal{R}$ cannot hold whenever $h_{min} > 2\lambda$.
\end{corollary}
 We have a list of functions that satisfy \cref{CF}
\begin{remark}
If $(M,g,f,m,\lambda)$ is a quasi-Einstein manifold that satisfies the Curvature identity, then $(M,g,f)$ is rigid if $h$ is one of the following that includes polynomial, exponential, logarithmic, and trigonometric functions of `$f$'.
\begin{center}
\begin{table}[h!]
\begin{tabular}{|c|}
\hline
\textbf{A list of functions} \\ \hline
$h = af^{k},\,\,a,k\,\in \mathbb{R}$\\
$h = ae^{f^{k}},\,k,a\in \mathbb{R} \,\,\text{and}\,\,a \leq 2\lambda$\\
$h = a log(f^{k}),\,\,a,k\,\in \mathbb{R}$\\
$h = a sin(f^{k}) + b cos(f^{k}) ,\,k,a,b\in \mathbb{R} \,\,\text{and}\,\,|a| + |b|\leq 2\lambda$\\
 \hline
\end{tabular}
\label{tab:onecol}
\end{table}
\end{center}
\end{remark}

\bibliographystyle{amsplain}
\bibliography{references}

\end{document}